\documentclass[11pt,reqno,a4paper]{amsart}
\usepackage{tikz}
\usepackage{pstricks}
\usepackage{amsmath}
\usepackage{amssymb}
\usepackage{graphicx}
\usepackage{amsthm}
\usepackage{enumerate}
\usepackage{cite}
\usepackage[mathscr]{eucal}
\usepackage{subfig}
\theoremstyle{plain}
\newtheorem{theorem}{Theorem}[section]

\newtheorem{proposition}[theorem]{Proposition}

\theoremstyle{definition}
\newtheorem{definition}[theorem]{Definition}
\newtheorem{example}[theorem]{Example}

\theoremstyle{remark}

\numberwithin{equation}{section}
\usepackage{tikz}
\usetikzlibrary{arrows.meta}
\usepackage{pgfplots}

\begin{document}

\setcounter {page}{1}
\title[Study of Existence and Stability of Fixed Points]{Study of Existence and Stability of Fixed Points of Hypotenuse Contracting Mappings with Applications to Parabolic PDE}

\author[K.\ Roy, A.\ Banerjee, L.K.\ Dey]
{Kushal Roy$^{1^*}$, Anish Banerjee$^{2}$, Lakshmi Kanta Dey$^{2}$}
\newcommand{\acr}{\newline\indent}

\address{{$^{1}$}   Kushal Roy,
                    Department of Mathematics, 
                    Dr. B.C. Roy Engineering College, 
                    Jemua Road, Fuljhore, Durgapur – 713206,
                    West Bengal, India} 
                    \email{kushal.roy93@gmail.com}

\address{{$^{2}$}   Anish Banerjee,
                    Department of Mathematics,
                    National Institute of Technology
                    Durgapur, West Bengal, India}
                    \email{anish22.mathematics@gmail.com}
\address{{$^{2}$}   Lakshmi Kanta Dey,
                    Department of Mathematics,
                    National Institute of Technology
                    Durgapur, West Bengal, India}
                    \email{lakshmikdey@yahoo.co.in}


\begin{abstract}
In this article, we introduce a novel class of mappings identified by their property of contracting the hypotenuse of a right-angled triangle in the setting of metric spaces. We establish sufficient conditions ensuring both the existence and uniqueness of fixed points. A geometric analysis, complemented by illustrative diagrams, is provided to differentiate these mappings from other familiar contraction types, namely perimeter and area contractions, supported by examples. Furthermore, we investigate the Ulam–Hyers stability of the associated fixed point equation, thereby strengthening the robustness of the theoretical framework. Finally, the derived results are applied to demonstrate the existence of solutions for a nonhomogeneous linear parabolic partial differential equation (PDE).

\end{abstract}

\maketitle

\section{Introduction and Preliminaries}
The foundational stone of fixed point analysis was laid by S. Banach. In 1922, through the establishment of the Banach contraction principle, the evolution of fixed point theory began. For the past century, researchers have studied and developed further results. The development is mainly achieved through the introduction of new contraction conditions and weakening of the structure of the underlying spaces.  

Most recently, in 2023, Petrov \cite{P23} introduced the notion of a novel class of mappings that contract the perimeters of triangles and established a result ensuring the existence of fixed points of these mappings. The definition of these mappings is provided as follows:
\begin{definition}[\cite{P23}]\label{D11}
Let $(X,d)$ be a metric space with $|X|\ge 3$. Then the mapping $T:X \to X$ is defined as mapping contracting perimeters of triangles if there exists an $\alpha \in [0,1)$ such that 
\begin{equation}\label{cpt}
d(Tx,Ty) +d(Ty,Tz) + d(Tz,Tx) 
\le \alpha [d(x,y) +d(y,z) + d(z,x)]
\end{equation}
for every three pairwise distinct points $x,y,z \in X$.
\end{definition}

Subsequently, Petrov and Bisht \cite{PB24} introduced generalized Kannan type mappings, P{\u{a}}curar and Popescu \cite{PP24} introduced generalized Chatterjea type mappings, and Bisht and Petrov \cite{BP24} introduced generalized Ćirić–Reich–Rus type mappings. Bey et al. \cite{BPS25} introduced generalized Edelstein type mappings. Jleli et al. introduced polynomial type contractions \cite{JPS25}. Roy introduced mappings that contract the axes of an ellipse \cite{R24}, mappings that contract the transverse axis of a hyperbola \cite{R26b}, and mappings that contract the area of geometrical figures \cite{R26a}. Through these articles, different mappings are explored, and a few results are obtained regarding the existence of fixed points of these mappings. For further details, see \cite{BMD2,BMD1,BMD3,BMDS26}.

\medskip
Let us now recall the definition of Ulam–Hyers stability for the fixed point equation.
\begin{definition}\cite{H41}
    Let $(X, d)$ be a metric space, and let $T:X \to X$ be a mapping. The fixed point equation $x=Tx$ is Ulam-Hyers stable if there is a constant $K > 0$ such that for each $\varepsilon >0$ and each $v \in X$ with $d(v,Tv) \leq \varepsilon$, there exists an $u \in X$ with $u=Tu$ such that $d(u,v) \leq K \varepsilon$.
\end{definition}
    
Motivated by the previous results, we aim to explore a new class of mappings that are characterized by their ability to contract the hypotenuse of a right-angled triangle. This class of mappings marks a distinct generalization, as it does not overlap with existing contractions. Furthermore, sufficient conditions are derived to guarantee both the existence and the uniqueness of fixed points. Subsequently, we provide a geometric interpretation of the proposed mapping to establish its distinction from other well-known mappings through pictorial representation. Furthermore, illustrative examples are provided to support the theoretical findings. Next, we investigate the Ulam-Hyers stability of the fixed point equation of the proposed mappings. Finally, we apply the theoretical findings to establish the existence of a solution for a nonhomogeneous linear parabolic partial differential equation. 



\section{Mappings contracting hypotenuse of a right-angled triangle}
We begin the section with the introduction of mappings that contract the hypotenuse of a right-angled triangle as follows:
\begin{definition}\label{21}
Let $(X,d)$ be a metric space with $|X|\ge 3$ and $T:X \to X$ be a mapping. Then $T$ is said to be a mapping contracting hypotenuse of a right-angled triangle if there exists $0 \le \beta<1$ such that the following inequality:
\begin{align}\label{e21}
[d(Tx,Ty)]^2+[d(Ty,Tz)]^2\le \beta ([d(x,y)]^2+[d(y,z)]^2),
\end{align}
holds for any three pairwise distinct points $x,y,z \in X$.
\end{definition}
The following result illustrates that a sub-collection of Banach contraction mappings is contained in the collection of mappings that contract the hypotenuse of a right-angled triangle.

\begin{proposition}\label{22}
    Let $(X, d)$ be a metric space with $|X| \geq 3$ and $T:X \to X$ be a Banach contraction mapping with contraction constant $k \in [0, \frac{1}{\sqrt{2}})$. Then $T$ is a mapping contracting hypotenuse of a right-angled triangle.
\end{proposition}

\begin{proof}
Since $T$ is a Banach contraction with contraction constant $k$, it follows that
\begin{align*}
    d(Tx,Ty) \leq k d(x,y) \text{ for all } x, y \in X.
\end{align*}
Therefore, for any three distinct elements $x, y, z \in X$, we have
\begin{align*}
    [d(Tx,Ty)]^2 +[d(Ty,Tz)]^2 \le 2{k}^2 ([d(x,y)]^2+ [d(y,z)]^2).
\end{align*}
As $k < \frac{1}{\sqrt{2}}$ so $2{k}^2 < 1$ and hence, $T$ is a mapping contracting hypotenuse of a right-angled triangle.
\end{proof}
In the following, we provide an example of a mapping contracting hypotenuse of a right-angled triangle, demonstrating its distinctness from some well-known classes of mappings:
\begin{example}\label{23}
Let $X$ be a nonempty set of cardinality greater than or equal to $3$. Choose two distinct elements $u,v \in X$ and consider two nonempty disjoint subsets $U$ and $V$ of $X$ containing $u$ and $v$ respectively such that $X=U \cup V$. Define a metric $d:X \times X \to [0,\infty)$ by
\begin{align*}
    d(x,y)=
    \begin{cases}
        &0, \text{ for }x=y,\\
        &1, \text{ for }x=u,y=v,\\
        &2, \text{ otherwise}.
    \end{cases}
\end{align*}
 Now let us define a mapping $T:X \to X$ by
\begin{align*}
    Tx=
    \begin{cases}
        &u, \text{ if }x \in U,\\
        &v, \text{ if }x \in V.
    \end{cases}
\end{align*}
 
Then for any three distinct elements $a,b,c$, we have the following cases:

\medskip
\noindent
\textbf{Case-I:} If $a,b,c \in U$, then it is obvious that $[d(Ta,Tb)]^2 + [d(Tb,Tc)]^2=0$.

\smallskip
\noindent
\textbf{Case-II:} If $a,b,c \in V$, then similarly $[d(Ta,Tb)]^2 + [d(Tb,Tc)]^2=0$.

\smallskip
\noindent
\textbf{Case-III:} If $a,b \in U$ and $c \in V$, then

\smallskip
$[d(Ta,Tb)]^2 + [d(Tb,Tc)]^2=[d(u,v)]^2=1$.

\smallskip
$[d(Ta,Tc)]^2 + [d(Tc,Tb)]^2=2[d(u,v)]^2=2$.

\smallskip
$[d(Tb,Ta)]^2 + [d(Ta,Tc)]^2=[d(u,v)]^2=1$ 

\smallskip
and  $[d(a,b)]^2 + [d(b,c)]^2 \geq 5$ for any $\{a,b,c\} \subseteq X$.

\smallskip
\noindent
\textbf{Case-IV:} If $a\in U$ and $b,c \in V$, then it follows analogously from Case III.\\
Therefore, it is evident that 
\begin{align*}
    [d(Tx,Ty)]^2+[d(Ty,Tz)]^2\le \frac{2}{5} ([d(x,y)]^2+[d(y,z)]^2), 
\end{align*}
for all pairwise distinct $x,y,z \in X$. Hence, $T$ is mapping contracting hypotenuse of a right-angled triangle. However, $T$ does not belong to any of the contraction classes of Banach, Kannan, Reich, Hardy-Rogers or \'Ciri\'c. 
\end{example}
Next, we check the continuity of a mapping contracting hypotenuse of a right-angled triangle.
\begin{theorem}\label{24}
In a metric space, a mapping contracting hypotenuse of a right-angled triangle is continuous. 
\end{theorem}

\begin{proof}
Let $(X,d)$ be a metric space with $|X|\ge 3$ and $T:X \to X$ be a mapping contracting hypotenuse of a right-angled triangle. Suppose that $x_0 \in X$ is arbitrary. If $x_0$ is an isolated point, then clearly $T$ is continuous at $x_0$. Otherwise, $x_0$ is an accumulation point. Now, for any $\delta>0$, $B(x_0,\delta)=\{x \in X: d(x,x_0)<\delta\}$ contains infinitely many points. So, consider two distinct points $x,y \in B(x_0,\delta)$ other than $x_0$. Now, from \eqref{e21} we have
\begin{align*}
[d(Tx,Tx_0)]^2
&\le [d(Tx,Tx_0)]^2 +[d(Tx_0,Ty)]^2\nonumber\\
&\le \beta\left([d(x,x_0)]^2 +[d(x_0,y)]^2\right)\nonumber\\
& <2\beta \delta^2\nonumber\\
\implies d(Tx,Tx_0) &<\sqrt{2\beta}\delta.
\end{align*}
This shows that for any $\varepsilon>0$ if we set $\delta<\frac{\varepsilon}{\sqrt{2\beta}}$, then 
$$d(Tx,Tx_0)<\varepsilon\text{ whenever }d(x,x_0)<\delta.$$
Thus, the result follows.
\end{proof}
We now present a result that establishes the existence of fixed points for mappings contracting hypotenuse of a right-angled triangle.

\begin{theorem}\label{25}
Let $T$ be a mapping contracting hypotenuse of a right-angled triangle defined over a complete metric space $(X,d)$ with $|X|\ge 3$. Then the fixed point set $Fix(T)$ is non-empty, provided $T$ does not possess any periodic point of prime period $2$.
\end{theorem}

\begin{proof}
Choose $x_0\in X$ arbitrarily and construct the Picard iterating sequence $x_n=Tx_{n-1}=T^n x_0$, $n\in \mathbb{N}$. If for some $n\ge 1$, $x_n=Tx_{n-1}$, then $T$ has a fixed point in $X$. Therefore, without loss of generality, we assume that $x_{n-1} \neq x_n$ for all $n \in \mathbb{N}$. Also, since $T$ does not have any periodic point of prime period $2$, it follows that $x_{n-1} \neq x_{n+1}$ for all $n \in \mathbb{N}$. This shows that for any $n \in \mathbb{N}$, $\{x_{n-1}, x_n, x_{n+1}\}$ is a set of pairwise distinct elements in $X$. Now, for any $n \in \mathbb{N}$, we have
\begin{align}\label{e24}
&[d(Tx_{n-1},Tx_{n})]^2 + [d(Tx_{n},Tx_{n+1})]^2 
\le \beta ([d(x_{n-1},x_{n})]^2+[d(x_{n},x_{n+1})]^2)\nonumber\\
\implies &[d(x_{n},x_{n+1})]^2+[d(x_{n+1},x_{n+2})]^2 
\le \beta ([d(x_{n-1},x_{n})]^2+[d(x_{n},x_{n+1})]^2)\nonumber\\
\implies &[d(x_{n+1},x_{n+2})]^2 
\le \beta ([d(x_{n-1},x_{n})]^2+[d(x_{n},x_{n+1})]^2)-[d(x_{n},x_{n+1})]^2\nonumber\\
\implies &[d(x_{n+1},x_{n+2})]^2 
\le \beta [d(x_{n-1},x_{n})]^2-(1-\beta)[d(x_{n},x_{n+1})]^2\nonumber\\
\implies &[d(x_{n+1},x_{n+2})]^2 \le \beta [d(x_{n-1},x_{n})]^2\nonumber\\
\implies &d(x_{n+1},x_{n+2}) \le \sqrt{\beta} d(x_{n-1},x_{n})=\sigma d(x_{n-1},x_{n}),\text{where }\sigma=\sqrt{\beta}.
\end{align}
From (\ref{e24}) we get, for all $m \in \mathbb{N}$ 
\begin{align*}
&d(x_{2m+1},x_{2m+2}) \le \sigma d(x_{2m-1},x_{2m}) \le \cdots \sigma^m d(x_1,x_2),\\
&d(x_{2m},x_{2m+1}) \le \sigma d(x_{2m-2},x_{2m-1}) \le \cdots \sigma^m d(x_0,x_1).
\end{align*}
For $n \in \mathbb{N}$ and $p=0,1,2,\cdots$, we have 
\begin{align}\label{e25}
d(x_n,x_{n+p})\le d(x_n,x_{n+1})+d(x_{n+1},x_{n+2})+\cdots+d(x_{n+p-1},x_{n+p}).
\end{align}
We now observe that four distinct cases arise, which will be examined in the following:

\smallskip
\textbf{Case I:} Let $n=2j$, $p=2r$ with $j \in \mathbb{N} \cup \{0\}$, $r \in \mathbb{N}$, then
\begin{align}\label{e26}
d(x_{2j},x_{2j+2r})
&\le d(x_{2j},x_{2j+1})+d(x_{2j+1},x_{2j+2})+\cdots\nonumber\\
&+d(x_{2j+2r-2},x_{2j+2r-1})+d(x_{2j+2r-1},x_{2j+2r})\nonumber\\
&\le \sigma^j d(x_0,x_1)+\sigma^j d(x_1,x_2)+\cdots+\sigma^{j+r-1} d(x_0,x_1)+\sigma^{j+r-1} d(x_1,x_2)\nonumber\\
&=(\sigma^j+\cdots+\sigma^{j+r-1})d(x_0,x_1)+(\sigma^j+\cdots+\sigma^{j+r-1}) d(x_1,x_2)\nonumber\\
&\le \sigma^j \frac{d(x_0,x_1)+d(x_1,x_2)}{1-\sigma}.
\end{align}

\textbf{Case II:} Let $n=2j$, $p=2r+1$ with $j \in \mathbb{N} \cup \{0\}$, $r \in \mathbb{N} \cup \{0\}$, then
\begin{align}\label{e27}
d(x_{2j},x_{2j+2r+1})
&\le d(x_{2j},x_{2j+1})+d(x_{2j+1},x_{2j+2})+\cdots\nonumber\\
&+d(x_{2j+2r-1},x_{2j+2r})+d(x_{2j+2r},x_{2j+2r+1})\nonumber\\
&\le \sigma^j d(x_0,x_1)+\sigma^j d(x_1,x_2)+\cdots+\sigma^{j+r-1} d(x_1,x_2)+\sigma^{j+r} d(x_0,x_1)\nonumber\\
&\le (\sigma^j+\cdots+\sigma^{j+r})d(x_0,x_1)+(\sigma^j+\cdots+\sigma^{j+r})d(x_1,x_2)\nonumber\\
&\le \sigma^j \frac{d(x_0,x_1)+d(x_1,x_2)}{1-\sigma}.
\end{align}

\textbf{Case III:} Let $n=2j+1$, $p=2r$ with $j \in \mathbb{N} \cup \{0\}$, $r \in \mathbb{N}$, then
\begin{align}\label{e28}
d(x_{2j+1},x_{2j+1+2r})
&\le d(x_{2j+1},x_{2j+2})+d(x_{2j+2},x_{2j+3})+\cdots\nonumber\\
&+d(x_{2j+2r-1},x_{2j+2r})+d(x_{2j+2r},x_{2j+2r+1})\nonumber\\
&\le \sigma^j d(x_1,x_2)+\sigma^{j+1} d(x_0,x_1)+\cdots+\sigma^{j+r-1} d(x_1,x_2)+\sigma^{j+r} d(x_0,x_1)\nonumber\\
&=(\sigma^j+\cdots+\sigma^{j+r})d(x_0,x_1)+(\sigma^j+\cdots+\sigma^{j+r})d(x_1,x_2)\nonumber\\
&\le \sigma^j \frac{d(x_0,x_1)+d(x_1,x_2)}{1-\sigma}.
\end{align}

\textbf{Case IV:} Let $n=2j+1$, $p=2r+1$ with $j \in \mathbb{N} \cup \{0\}$, $r \in \mathbb{N} \cup \{0\}$, then
\begin{align}\label{e29}
d(x_{2j+1},x_{2j+2r+2})
&\le d(x_{2j+1},x_{2j+2})+d(x_{2j+2},x_{2j+3})+\cdots\nonumber\\
&+d(x_{2j+2r},x_{2j+2r+1})+d(x_{2j+2r+1},x_{2j+2r+2})\nonumber\\
&\le \sigma^j d(x_1,x_2)+\sigma^{j+1} d(x_0,x_1)+\cdots+\sigma^{j+r} d(x_0,x_1)+\sigma^{j+r} d(x_1,x_2)\nonumber\\
&=(\sigma^j+\cdots+\sigma^{j+r})d(x_0,x_1)+(\sigma^j+\cdots+\sigma^{j+r})d(x_1,x_2)\nonumber\\
&\le \sigma^j \frac{d(x_0,x_1)+d(x_1,x_2)}{1-\sigma}.
\end{align}
From \eqref{e25}--\eqref{e29}, we have $\displaystyle \lim_{n \to \infty} d(x_n,x_{n+p})=0$. Thus, $\{x_n\}$ is a Cauchy sequence in $X$. Since $X$ is complete, there exists $u\in X$ such that $x_n \to u$ as $n \to \infty$. Since $T$ is continuous, it follows that $x_{n+1}=Tx_n\to Tu$ whenever $n\to \infty.$ Thus, due to the completeness of $(X,d)$, we have $Tu=u$ and consequently $u$ is a fixed point of $T$. 
\end{proof}
The subsequent result addresses the cardinality of the fixed point set.
\begin{theorem}\label{26}
Let $T$ be a mapping contracting hypotenuse of a right-angled triangle defined over a metric space $(X,d)$ with $|X|\ge 3$. Then, the fixed point set $Fix(T)$ is non-empty and contains at most two elements.
\end{theorem}

\begin{proof}
Theorem \ref{25} confirms that $T$ has a fixed point and therefore, $Fix(T)$ is non-empty. If possible, suppose that $Fix(T)$ contains three distinct elements $u,v$ and $w$. Then, from  \eqref{21}, we see that
\begin{align*}
& [d(Tu,Tv)]^2+[d(Tv,Tw)]^2 \le [d(u,v)]^2+[d(v,w)]^2\\
\implies &[d(u,v)]^2+[d(v,w)]^2 \le [d(u,v)]^2+[d(v,w)]^2\\
\implies &(1-\beta)([d(u,v)]^2+[d(v,w)]^2) \le 0\\
\implies &[d(u,v)]^2+[d(v,w)]^2 =0\\
\implies &d(u,v)=d(v,w) =0, \text{ leads to a contradiction}.
\end{align*}
Hence, the result follows.
\end{proof}
Next, we present an illustrative example of a mapping contracting hypotenuse of a right-angled triangle in support of Theorem \ref{25}.

\begin{example}\label{27}
Consider $(\mathbb{N},d)$ as a standard metric space. Define $T:\mathbb{N}\to \mathbb{N}$ by
\begin{align*}
Tx=   
\begin{cases}
    \frac{x}{2}, &\text{ if $x$ is even},\\
    \frac{x+1}{2}, &\text{ if $x$ is odd}. 
\end{cases}
\end{align*}

To show that $T$ is a mapping contracting hypotenuse of a right-angled triangle, we consider three distinct points $x,y$ and $z$ and for different possible cases we will ensure whether the ratio $R=\frac{[d(Tx,Ty)]^2 +[d(Ty,Tz)]^2}{[d(x,y)]^2 +[d(y,z)]^2}$ (from Definition \ref{21}) lies on $[0,1)$. Suppose that $x - y=a$, $y - z=b$ and $z - x=c$. Now, the following cases occur:\\
\textbf{Case I:} If $x,y,z$ are all even, then there are three sub-cases depending on the choices of base and height.\\
\textbf{Sub-case I:} Taking $xy$ and $yz$ as the base and height, we get 
 \begin{align}\label{2811}
    R=\frac{[d(Tx,Ty)]^2 +[d(Ty,Tz)]^2}{[d(x,y)]^2 +[d(y,z)]^2}
     =\frac{(\frac{x-y}{2})^2 +(\frac{y-z}{2})^2}{(x-y)^2 +(y-z)^2}
     =\frac{1}{4}.
\end{align}
\textbf{Sub-case II:} Taking $xz$ and $zy$ as the base and height, we get 
\begin{align}\label{2812}
    R=\frac{[d(Tx,Tz)]^2 +[d(Tz,Ty)]^2}{[d(x,z)]^2 +[d(z,y)]^2}
     =\frac{(\frac{x-z}{2})^2 +(\frac{z-y}{2})^2}{(x-z)^2 +(z-y)^2}
     =\frac{1}{4}.
\end{align}
\textbf{Sub-case III:} Taking $yx$ and $xz$ as the base and height, we get 
\begin{align}\label{2813}
    R=\frac{[d(Ty,Tx)]^2 +[d(Tx,Tz)]^2}{[d(y,x)]^2 +[d(x,z)]^2}
     =\frac{(\frac{y-x}{2})^2 +(\frac{x-z}{2})^2}{(y-x)^2 +(x-z)^2}
     =\frac{1}{4}.
\end{align}   
\textbf{Case II:} If $x$ and $y$ are even and $z$ is odd, then there are three sub-cases depending on the choices of base and height.\\
\textbf{Sub-case I:} Taking $xy$ and $yz$ as the base and height, we have
\begin{align*}
    R=\frac{[d(Tx,Ty)]^2 +[d(Ty,Tz)]^2}{[d(x,y)]^2 +[d(y,z)]^2}
     =\frac{(\frac{x-y}{2})^2 +(\frac{y-z-1}{2})^2}{(x-y)^2 +(y-z)^2}.
\end{align*}
Now, substituting $a$ and $b$ in $R$, we have
\begin{align*}
    R= \frac{a^2 + (b-1)^2}{4(a^2 + b^2)}
     =\frac{a^2 + b^2 - 2b + 1}{4(a^2 + b^2)}
     = \frac{1}{4} + \frac{1 - 2b}{4(a^2 + b^2)}.
\end{align*}
Since $a,b \in \mathbb{N}$, we get
\begin{align*}
    a^2 + (b-1)^2 \ge 1
    \implies \frac{-2b}{a^2 + b^2} \le 1
    \implies \frac{1 - 2b}{a^2 + b^2} \le \frac{1}{2} + 1
    \implies \frac{1 - 2b}{4(a^2 + b^2)} \le \frac{3}{8}.
\end{align*}
Therefore, 
\begin{align}\label{2821}
   R=\frac{(\frac{x-y}{2})^2 +(\frac{y-z-1}{2})^2}{(x-y)^2 +(y-z)^2} \le \frac{5}{8}. 
\end{align}
\textbf{Sub-case II:} Taking $xz$ and $zy$ as the base and height, we get
\begin{align*}
    R=\frac{[d(Tx,Tz)]^2 +[d(Tz,Ty)]^2}{[d(x,z)]^2 +[d(z,y)]^2}
     =\frac{(\frac{x-z-1}{2})^2 +(\frac{z-y+1}{2})^2}{(x-z)^2 +(z-y)^2}.
\end{align*}
Now, substituting $b$ and $c$ in $R$, we have
\begin{align*}
    R 
     = \frac{(c+1)^2 + (b-1)^2}{4(c^2 + b^2)}
     =\frac{c^2 + b^2 + 2(c - b) + 2}{4(c^2 + b^2)}
     = \frac{1}{4} + \frac{2(c - b) + 2}{4(c^2 + b^2)}.
\end{align*}
Since $b,c \in \mathbb{N}$, we get
\begin{align*}
    (c-1)^2 + (b+1)^2 \ge 4
    \implies 2(c - b) + 2 \le c^2 + b^2
    \implies \frac{2(c - b) + 2}{4(c^2 + b^2)} \le \frac{1}{4}.
\end{align*}
Therefore, 
\begin{align}\label{2822}
   R= \frac{(\frac{x-z-1}{2})^2 +(\frac{z-y+1}{2})^2}{(x-z)^2 +(z-y)^2} \le \frac{1}{2}. 
\end{align}
\textbf{Sub-case III:} Taking $yx$ and $xz$ as the base and height, we have
\begin{align*}
    R=\frac{[d(Ty,Tx)]^2 +[d(Tx,Tz)]^2}{[d(y,x)]^2 +[d(x,z)]^2}
     =\frac{(\frac{y-x}{2})^2 +(\frac{x-z-1}{2})^2}{(y-x)^2 +(x-z)^2}.
\end{align*}
Now, substituting $a$ and $c$ in $R$, we have
\begin{align*}
    R= \frac{a^2 + (c+1)^2}{4(a^2 + c^2)}
     =\frac{a^2 + c^2 + 2c + 1}{4(a^2 + c^2)}
     = \frac{1}{4} + \frac{1 + 2c}{4(a^2 + c^2)}.
\end{align*}
Since $a,c \in \mathbb{N}$, we get
\begin{align*}
    a^2 + (c-1)^2 \ge 1
    \implies \frac{2c}{a^2 + c^2} \le 1
    \implies \frac{1 + 2c}{a^2 + c^2} \le \frac{1}{2} + 1
    \implies \frac{1 + 2c}{4(a^2 + c^2)} \le \frac{3}{8}.
\end{align*}
Therefore, 
\begin{align}\label{2823}
   R= =\frac{(\frac{y-x}{2})^2 +(\frac{x-z-1}{2})^2}{(y-x)^2 +(x-z)^2} \le \frac{5}{8}. 
\end{align}

Following this method, six additional cases arise based on whether $x, y,$ and $z$ are even or odd. Each of these cases further splits into three sub-cases depending on the choice of base and height. We have checked all the cases and observed that 
\begin{align*}
    [d(Tx,Ty)]^2+[d(Ty,Tz)]^2\le \frac{5}{8} ([d(x,y)]^2+[d(y,z)]^2),
\end{align*}
for all pairwise distinct $x,y,z \in \mathbb{N}$. Hence, $T$ is mapping contracting hypotenuse of a right-angled triangle with $\beta \in [\frac{5}{8},1)$. In addition, $T$ does not possess any periodic point of prime period $2$. Thus, by Theorem \ref{25}, $T$ has a fixed point. Here, $\{1\}$ is the unique fixed point of $T$. However, it should be noted that $T$ does not satisfy any of the contractive conditions of Banach, Kannan, Chatterjea, or Reich.

\end{example}

Now, we provide an example of a mapping contracting hypotenuse of a right-angled triangle with two fixed points.
\begin{example}\label{28}
Let us consider $\mathbb{N}$ equipped with the Euclidean metric $d$. Define a mapping $T:\mathbb{N}\to \mathbb{N}$ by
\begin{align*}
Tx=   
\begin{cases}
    &1, \text{ for }x=1,\\
    &2, \text{ for }x \ge 2. 
\end{cases}
\end{align*}

Then for three distinct points $x,y$ and $z$, we have the following cases:\\
\textbf{Case-I:} If $x=1$ and $y,z\ge 2$, then 

\begin{align*}
[d(T1,Ty)]^2+[d(Ty,Tz)]^2=&[(1-2)]^2+[(2-2)]^2=1\\ 
\text {and}\\
[d(1,y)]^2+[d(y,z)]^2=&[(1-y)]^2+[(y-z)]^2\ge 2. 
\end{align*}

\begin{align*}
[d(T1,Tz)]^2+[d(Tz,Ty)]^2=&[(1-2)]^2+[(2-2)]^2=1\\ 
\text {and}\\
[d(1,z)]^2+[d(z,y)]^2=&[(1-z)]^2+[(z-y)]^2\ge 2. 
\end{align*}

\begin{align*}
[d(Ty,T1)]^2+[d(T1,Tz)]^2=&[(2-1)]^2+[(1-2)]^2=2\\ 
\text {and}\\
[d(y,1)]^2+[d(1,z)]^2=&[(1-y)]^2+[(1-z)]^2\ge 5. 
\end{align*}
\textbf{Case-II:} If $x,y,z \ge 2$, then $[d(Tx,Ty)]^2+[d(Ty,Tz)]^2=0$ for any $x,y,z \in \mathbb{N}$.\\
Therefore it is evident that 
\begin{align*}
    [d(Tx,Ty)]^2+[d(Ty,Tz)]^2\le \frac{1}{2} ([d(x,y)]^2+[d(y,z)]^2),
\end{align*}
for all pairwise distinct $x,y,z \in \mathbb{N}$. Hence, $T$ is a mapping contracting hypotenuse of a right-angled triangle which is not an usual contractive mapping. Moreover, it satisfies all the conditions of Theorem \ref{25} and possesses exactly two fixed points $1$ and $2$.
\end{example}

In Example~\ref{28}, we observed that a mapping contracting hypotenuse of a right-angled triangle may admit multiple fixed points. Conversely, one can construct a scenario in which the uniqueness of the fixed point is guaranteed for such mappings.

\begin{theorem}\label{29}
Let $T$ be a mapping contracting hypotenuse of a right-angled triangle defined over a metric space $(X,d)$ with $|X|\ge 3$ and $T$ does not contain periodic points of prime period 2. Assume that $X$ contains infinitely many points such that for $x_0 \in X$, the iterative sequence 
\begin{align*}
   \{ x_0, \, x_1 = Tx_0, \, x_2 = Tx_1, \dots \}
\end{align*}
converges to a point $u \in X$ with $u \neq x_n$ for all $n \in \mathbb{N} \cup \{0\}$. Then $u$ is the unique fixed point of $T$.
\end{theorem}

\begin{proof}
From Theorem~\ref{25}, it follows that $u$ is a fixed point of $T$. Let $v$ be another fixed point of $T$. Then $v \neq x_n$, for all $n \in \mathbb{N} \cup \{0\}$, otherwise we have $u = v$. Therefore, $u, v$ and $x_n$ are distinct for any $n \in \mathbb{N} \cup \{0\}$. Now, for all $n \in \mathbb{N}\cup \{0\}$, we get
\begin{align*}
R_n&= \dfrac{{[d(Tu, Tx_n)]}^2 + {[d(Tx_n, Tv)]}^2}{{[d(u, x_n)]}^2 + {[d(x_n, v)]}^2}\\
&= \dfrac{{[d(u, x_{n+1})]}^2 + {[d(x_{n+1}, v)]}^2}{{[d(u, x_n)]}^2 + {[d(x_n, v)]}^2}.
\end{align*}
Now, taking $n \to \infty$, we get $R_n \to 1$, which leads to a contradiction to the fact that $R_n\le 1$ for all $n \in \mathbb{N} \cup \{0\}$ from (\ref{e21}). Hence, $Fix(T)= \{u\}$.
\end{proof}
The section concludes with a theorem that establishes Ulam–Hyers stability for the fixed point equation arising from mappings contracting hypotenuse of a right-angled triangle.

\begin{theorem}\label{210}
    Let $T$ be a mapping contracting hypotenuse of a right-angled triangle defined over a metric space $(X,d)$ with $|X| \ge 3$. If $T$ has two fixed points in $X$, then the fixed point equation $x=Tx$ is Ulam--Hyers stable.
\end{theorem}

\begin{proof}
    Let $x$ and $y$ be two distinct fixed points of $T$. Then for any $p \notin \{x, y\}$, we have
\begin{align*}
     [d(Tp,x)]^2 + [d(x,y)]^2 &= [d(Tp,Tx)]^2 + [d(Tx,Ty)]^2\\ 
    &\leq \beta ([d(p,x)]^2 + [d(x,y)]^2)\\
    &\leq \beta [d(p,x)]^2 + [d(x,y)]^2\\
    \implies [d(Tp,x)]^2 &\leq \beta [d(p,x)]^2\\
    \implies d(Tp,x) &\leq \sqrt{\beta} d(p,x).
\end{align*}

\noindent 
Therefore, for any $\epsilon > 0$, let $z \in X$ be an $\varepsilon$-solution, i.e., $$d(z,Tz) \leq \varepsilon.$$
Now, we have
\begin{align*}
    d(z,x) &\le d(z,Tz) + d(Tz,x) \\
    &\le d(z,Tz) + \sqrt{\beta} d(z,x) \\
    \implies (1- \sqrt{\beta}) d(z,x) &\le d(z,Tz)\\
    \implies d(z,x) &\le \frac{1}{1- \sqrt{\beta}} d(z,Tz) \le \frac{\varepsilon}{1- \sqrt{\beta}}.
\end{align*}
which implies that
\begin{align*}
    d(z,x) \le K \varepsilon
\end{align*}
where $K=\frac{1}{1- \sqrt{\beta}}$. Hence, the result follows.
\end{proof}

\section{Geometrical viewpoint of the mapping contracting hypotenuse of a right-angled triangle}



This section develops a geometric interpretation of the proposed mapping and rigorously establishes its distinction from both perimetric contraction mappings and area contraction mappings. \\
Consider a right-angled triangle with base $t$ and height $r$, whose hypotenuse length is 
\[
L = \sqrt{r^2 + t^2}.
\]In the Euclidean space $(\mathbb{R}^2,d)$, suppose a mapping $T$ contracts the hypotenuse of a right-angled triangle $\triangle x_1x_2x_3$ with contraction constant $\beta \in [0,1)$. If the image under $T$ is again a right-angled triangle $\triangle Tx_1Tx_2Tx_3$, then we obtain
\[
L_2 = [d(Tx_1,Tx_2)]^2 + [d(Tx_2,Tx_3)]^2 
     \leq \beta \big( [d(x_1,x_2)]^2 + [d(x_2,x_3)]^2 \big) 
     = \beta L_1 < L_1.
\]
Figure 1 provides a pictorial representation of this mapping.
\newpage
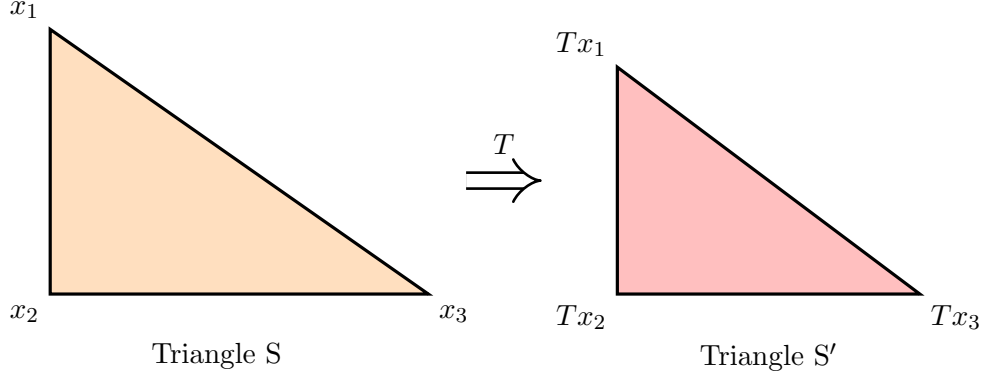
\begin{figure}[htbp]
\centering
\begin{tikzpicture}[line width=1pt, >=Latex]

    \coordinate (x2) at (0,0);
    \coordinate (x1) at (0,3.5);
    \coordinate (x3) at (5,0);
    
    \filldraw[
    fill=orange!25,
    draw=black,
    very thick
] (x1) -- (x2) -- (x3) -- cycle;
    
    \node[above left] at (x1) {$x_1$};
    \node[below left] at (x2) {$x_2$};
    \node[below right] at (x3) {$x_3$};
    
    \node[below=0.5cm] at (2.2, 0) {Triangle S};

    \draw[line width=1pt, double, double distance=5pt, -{Implies[]}] (5.5, 1.5) -- (6.5, 1.5);
    \node[above=0.2cm] at (6.0, 1.5) {$T$};

    \coordinate (Tx2) at (7.5,0);
    \coordinate (Tx1) at (7.5,3.0);
    \coordinate (Tx3) at (11.5,0);
    
    \filldraw[
    fill=red!25,
    draw=black,
    very thick
] (Tx1) -- (Tx2) -- (Tx3) -- cycle;
    
    \node[above left] at (Tx1) {$Tx_1$};
    \node[below left] at (Tx2) {$Tx_2$};
    \node[below right] at (Tx3) {$Tx_3$};
    
    \node[below=0.5cm] at (9.5, 0) {Triangle S$'$};

\end{tikzpicture}
\caption{Image of a right-angled triangle under the mapping $T$}
\end{figure}


We now turn to a specific example demonstrating that hypotenuse contraction  does not necessarily imply a contraction of perimeter or area of triangles.
\begin{example}
Consider 
\begin{align*}
    X=\{x_1,x_2,x_3,y_1,y_2,y_3\}.
\end{align*}
Define a metric $d$ on $X$ by
\begin{align*}
d(x_1,x_2)&=2,\qquad
d(x_2,x_3)=2,\qquad
d(x_3,x_1)=\frac{1}{5},\\
d(y_1,y_2)&=\frac{8}{5},\qquad
d(y_2,y_3)=\frac{8}{5},\qquad
d(y_3,y_1)=\frac{11}{10},
\end{align*}
and
\begin{align*}
    d(x_i,y_j)=2, \text{ for } i,j\in\{1,2,3\}.
\end{align*}
The remaining values of $d$ are determined by symmetry and $d(x,x)=0$. It is evident that $(X,d)$ is a metric space.

Define $T:X\to X$ by
\begin{align*}
T(x_1)=y_1,\qquad
T(x_2)=y_2,\qquad
T(x_3)=y_3, 
\end{align*}
and
\begin{align*}
    T(y_1)=T(y_2)=T(y_3)=y_1.
\end{align*}

To establish that $T$ is a mapping contracting hypotenuse of a right-angled triangle, we begin by selecting three distinct points $x, y, z \in X$. For each admissible triplet of these points, it is necessary to verify whether the ratio
\begin{align*}
    R=\frac{[d(Tx,Ty)]^2 +[d(Ty,Tz)]^2}{[d(x,y)]^2 +[d(y,z)]^2}
\end{align*}
satisfies $R \in [0,1)$. Since the total number of possible triplets is 120 (with half of them being identical, leaving 60 distinct cases), we proceed in an orderly manner, beginning with the initial triplet $(x_{1}, x_{2}, x_{3})$. The analysis for the remaining triplets follows in the same fashion.
\begin{align*}
[d(Tx_1,Tx_2)]^2+[d(Tx_2,Tx_3)]^2
=\left(\frac{8}{5}\right)^2+
  \left(\frac{8}{5}\right)^2
=\frac{128}{25},
\end{align*}
whereas
\begin{align*}
    [d(x_1,x_2)]^2+[d(x_2,x_3)]^2 =4+4=8.
\end{align*}
Thus,
\begin{align*}
    \frac{[d(Tx_1,Tx_2)]^2+[d(Tx_2,Tx_3)]^2}{[d(x_1,x_2)]^2+[d(x_2,x_3)]^2}=\frac{16}{25}<1.
\end{align*}
On the other hand,
\begin{align*}
[d(Tx_2,Tx_3)]^2+[d(Tx_3,Tx_1)]^2
=\left(\frac{8}{5}\right)^2+
  \left(\frac{11}{10}\right)^2
=\frac{64}{25}+\frac{121}{100}
=\frac{377}{100},
\end{align*}
while
\begin{align*}
[d(x_2,x_3)]^2+[d(x_3,x_1)]^2
=4+\frac{1}{25}
=\frac{101}{25}.
\end{align*}
Consequently,
\begin{align*}
    \frac{[d(Tx_2,Tx_3)]^2+[d(Tx_3,Tx_1)]^2}{[d(x_2,x_3)]^2+[d(x_3,x_1)]^2}
    =\frac{377}{404}<1.
\end{align*}
Next, we have
\begin{align*}
[d(Tx_2,Tx_1)]^2+[d(Tx_1,Tx_3)]^2
=\left(\frac{8}{5}\right)^2 + \left(\frac{11}{10}\right)^2
=\frac{377}{404},
\end{align*}
whereas
\begin{align*}
   [d(x_2,x_1)]^2+[d(x_1,x_3)]^2 =4+\frac{1}{25} =\frac{101}{25}. 
\end{align*}
Thus,
\begin{align*}
    \frac{[d(Tx_2,Tx_1)]^2+[d(Tx_1,Tx_3)]^2}{[d(x_2,x_1)]^2+[d(x_1,x_3)]^2}
    =\frac{377}{404}<1.
\end{align*}
Upon checking all cases, it is observed that  
\begin{align*}
    \sup_{\substack{x,y,z\in X\\
x \neq y \neq z}}
\frac{[d(Tx,Ty)]^2+[d(Ty,Tz)]^2}{[d(x,y)]^2+[d(y,z)]^2}
=\frac{377}{404}.
\end{align*}
Therefore, we have
\begin{align*}
    [d(Tx,Ty)]^2+[d(Ty,Tz)]^2 \leq \beta \left([d(x,y)]^2+[d(y,z)]^2 \right),
\end{align*}
for every three pairwise distinct points $x,y,z\in X$ with $\beta \in [\frac{377}{404}, 1)$.\\
Therefore, $T$ is a mapping contracting hypotenuse of a right-angled triangle.

However, $T$ is not a mapping contracting perimeters of triangles.
Indeed, consider the triangle $(x_1,x_2,x_3)$. Its perimeter is
\begin{align*}
P(x_1,x_2,x_3)
=d(x_1,x_2)+d(x_2,x_3)+d(x_3,x_1)
=2+2+\frac{1}{5}
=\frac{21}{5}.
\end{align*}
The perimeter of its image triangle is
\begin{align*}
P(Tx_1,Tx_2,Tx_3)
=d(y_1,y_2)+d(y_2,y_3)+d(y_3,y_1)
=\frac{8}{5}+\frac{8}{5}+\frac{11}{10}
=\frac{43}{10}.
\end{align*}
Thus, we have
\begin{align*}
    P(Tx_1,Tx_2,Tx_3)>P(x_1,x_2,x_3).
\end{align*}
Therefore, there does not exist any $\alpha<1$ such that
\begin{align*}
    d(Tx,Ty)+d(Ty,Tz)+d(Tz,Tx) \leq \alpha [d(x,y)+d(y,z)+d(z,x)],
\end{align*}
for all pairwise distinct $x,y,z\in X$.

Even, $T$ is not a mapping contracting area of a right-angled triangle.
Indeed, consider the triangle $(x_2,x_1,x_3)$. Its area is
\begin{align*}
A(x_2,x_1,x_3)
= \frac{1}{2} \ d(x_2,x_1)\ d(x_1,x_3)
=\frac{1}{5}.
\end{align*}

The area of its image triangle is
\begin{align*}
A(Tx_2,Tx_1,Tx_3)
= \frac{1}{2}\ d(Tx_2,Tx_1)\ d(Tx_1,Tx_3)
=\frac{1}{2} d(y_2,y_1)\ d(y_1,y_3)
=\frac{22}{25}.
\end{align*}
Thus, we have
\begin{align*}
   A(Tx_2,Tx_1,Tx_3) > A(x_2,x_1,x_3). 
\end{align*}
Therefore, there does not exist any $\alpha<1$ such that
\begin{align*}
    d(Tx,Ty)d(Ty,Tz) \leq \alpha d(x,y)d(y,z),
\end{align*}
for all pairwise distinct $x,y,z\in X$.

Thus, $T$ is a mapping contracting hypotenuse of a right-angled triangle which is neither a mapping contracting perimeters of triangles nor a mapping contracting area of a right-angled triangle.
\end{example}
\newpage
Figure 2 presents an pictorial illustration of a right‑angled triangle whose hypotenuse is contracted under the given mapping, yet the perimeter of the resulting image triangle is unexpectedly larger. 
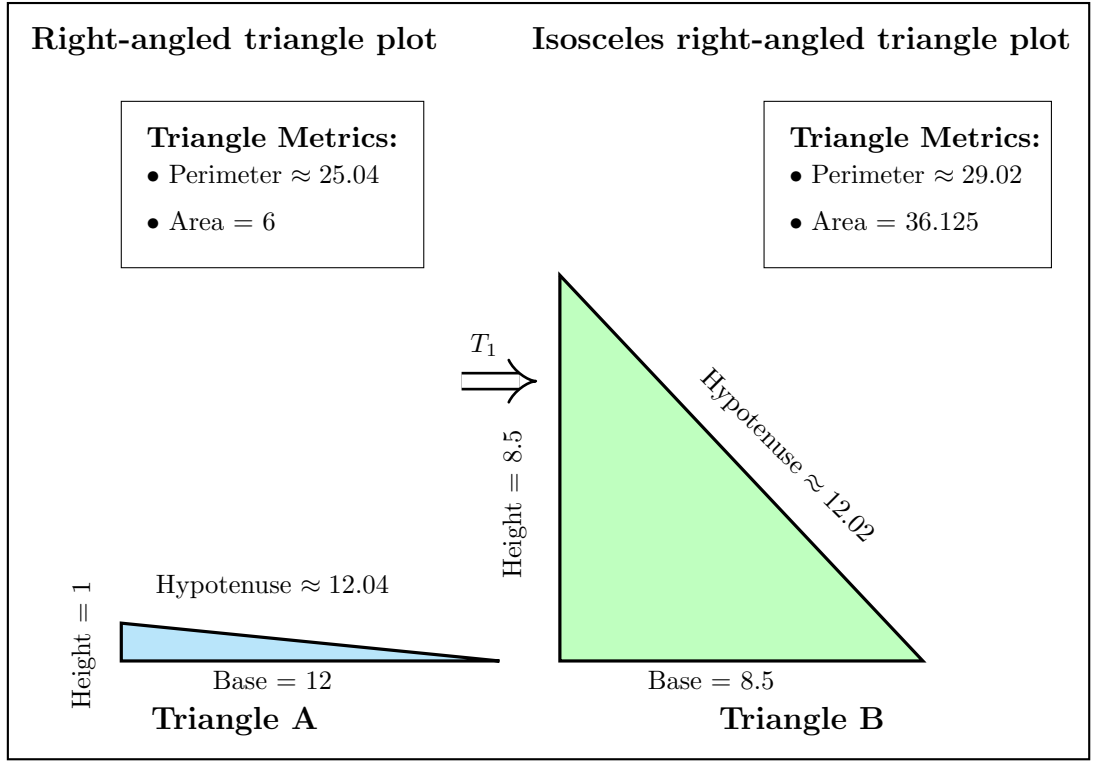
\begin{figure}[htbp]
\centering
\begin{tikzpicture}[font=\small]

\draw[thick] (-0.5,-0.5) rectangle (13.8,9.5);


\node[font=\bfseries\large] at (2.5,9)
{Right-angled triangle plot};

\node[font=\bfseries\large] at (10,9)
{Isosceles right-angled triangle plot};


\draw (1,6) rectangle (5,8.2);

\node[anchor=north west,font=\bfseries]
at (1.2,8.0){Triangle Metrics:};

\node[anchor=north west,align=left]
at (1.2,7.5)
{
$\bullet$ Perimeter $\approx25.04$\\[2mm]
$\bullet$ Area = 6
};


\coordinate (A) at (1,0.8);
\coordinate (B) at (6,0.8);
\coordinate (C) at (1,1.3);
\filldraw[
    fill=cyan!25,
    draw=black,
    very thick
] (A)--(B)--(C)--cycle;

\node[below] at (3,0.8)
{Base = 12};

\node[rotate=90,left]
at (0.5,2)
{Height = 1};

\node[above]
at (3,1.5)
{Hypotenuse $\approx12.04$};

    \draw[line width=1pt, double, double distance=5pt, -{Implies[]}] (5.5, 4.5) -- (6.5, 4.5);
    \node[above=0.2cm] at (5.8, 4.5) {$T_1$};


\draw (9.5,6) rectangle (13.3,8.2);

\node[anchor=north west,font=\bfseries]
at (9.7,8.0){Triangle Metrics:};

\node[anchor=north west,align=left]
at (9.7,7.5)
{
$\bullet$ Perimeter $\approx29.02$\\[2mm]
$\bullet$ Area = 36.125
};


\coordinate (P) at (6.8,0.8);
\coordinate (Q) at (11.6,0.8);
\coordinate (R) at (6.8,5.9);
\filldraw[
    fill=green!25,
    draw=black,
    very thick
] (P)--(Q)--(R)--cycle;

\node[below]
at (8.8,0.8)
{Base = 8.5};

\node[rotate=90,left]
at (6.2,4)
{Height = 8.5};

\node[rotate=-45]
at (9.8,3.5)
{Hypotenuse $\approx12.02$};

\node[font=\bfseries\large] at (2.5,0)
{Triangle A};

\node[font=\bfseries\large] at (10,0)
{Triangle B};
\end{tikzpicture}

\begin{equation*}
\text{Triangle } A \; \rightarrow \; 
\begin{array}{|c|}
\hline
T_1\colon \mathbb{R}^2 \rightarrow \mathbb{R}^2 \text{ defined by} \\
T_1(x, y) = \left( \frac{17}{24}x, \frac{17}{2}y \right), \, x, y \in \mathbb{R} \\
\hline
\end{array}
\; \rightarrow \; \text{Triangle } B
\end{equation*}
\caption{Images of two triangles $A$ and $B$ such that the hypotenuse of $A$ is contracted to the hypotenuse of $B$ under the mapping $T_1$, while the perimeter of triangle $A$ remains smaller than that of triangle $B$}
\end{figure}
\newpage
Finally, Figure 3 presents an pictorial illustration of a right‑angled triangle whose hypotenuse is contracted under the given mapping, yet the area of the resulting image triangle is unexpectedly increased. 

\begin{figure}[htbp]
\centering
\begin{tikzpicture}[font=\small]

\draw[thick] (-0.5,-0.5) rectangle (13.8,9.5);


\node[font=\bfseries\large] at (2.5,9)
{Right-angled triangle plot};

\node[font=\bfseries\large] at (10,9)
{Right-angled triangle plot};


\draw (1,6) rectangle (5,8.2);

\node[anchor=north west,font=\bfseries]
at (1.2,8.0){Triangle Metrics:};

\node[anchor=north west,align=left]
at (1.2,7.5)
{
$\bullet$ Perimeter $\approx21.05$\\[2mm]
$\bullet$ Area = 5
};


\coordinate (A) at (1,0.8);
\coordinate (B) at (6,0.8);
\coordinate (C) at (1,1.3);
\filldraw[
    fill=purple!25,
    draw=black,
    very thick
] (A)--(B)--(C)--cycle;

\node[below] at (3,0.8)
{Base = 10};

\node[rotate=90,left]
at (0.5,2)
{Height = 1};

\node[above]
at (3,1.5)
{Hypotenuse $\approx10.05$};

    \draw[line width=1pt, double, double distance=5pt, -{Implies[]}] (5.5, 4.5) -- (6.5, 4.5);
    \node[above=0.2cm] at (5.8, 4.5) {$T_2$};


\draw (9.5,6) rectangle (13.3,8.2);

\node[anchor=north west,font=\bfseries]
at (9.7,8.0){Triangle Metrics:};

\node[anchor=north west,align=left]
at (9.7,7.5)
{
$\bullet$ Perimeter = 24\\[2mm]
$\bullet$ Area = 24
};


\coordinate (P) at (7,0.8);
\coordinate (Q) at (10,0.8);
\coordinate (R) at (7,4.8);
\filldraw[
    fill=violet!25,
    draw=black,
    very thick
] (P)--(Q)--(R)--cycle;

\node[below]
at (8.8,0.8)
{Base = 6};

\node[rotate=90,left]
at (6.5,4)
{Height = 8};

\node[rotate=-50]
at (9,3.5)
{Hypotenuse $=10$};

\node[font=\bfseries\large] at (3,0)
{Triangle C};

\node[font=\bfseries\large] at (8.8,0)
{Triangle D};
\end{tikzpicture}
\begin{equation*}
\text{Triangle } C \; \rightarrow \; 
\begin{array}{|c|}
\hline
T_2\colon \mathbb{R}^2 \rightarrow \mathbb{R}^2 \text{ defined by} \\
T_2(x, y) = \left( \frac{3}{5}x, 8y \right), \, x, y \in \mathbb{R} \\
\hline
\end{array}
\; \rightarrow \; \text{Triangle } D
\end{equation*}
\caption{Illustration of two triangles $C$ and $D$ where the hypotenuse of $C$ is contracted to that of $D$ under the mapping $T_2$, while the area of triangle $C$ remains smaller than the area of triangle $D$}
\end{figure}
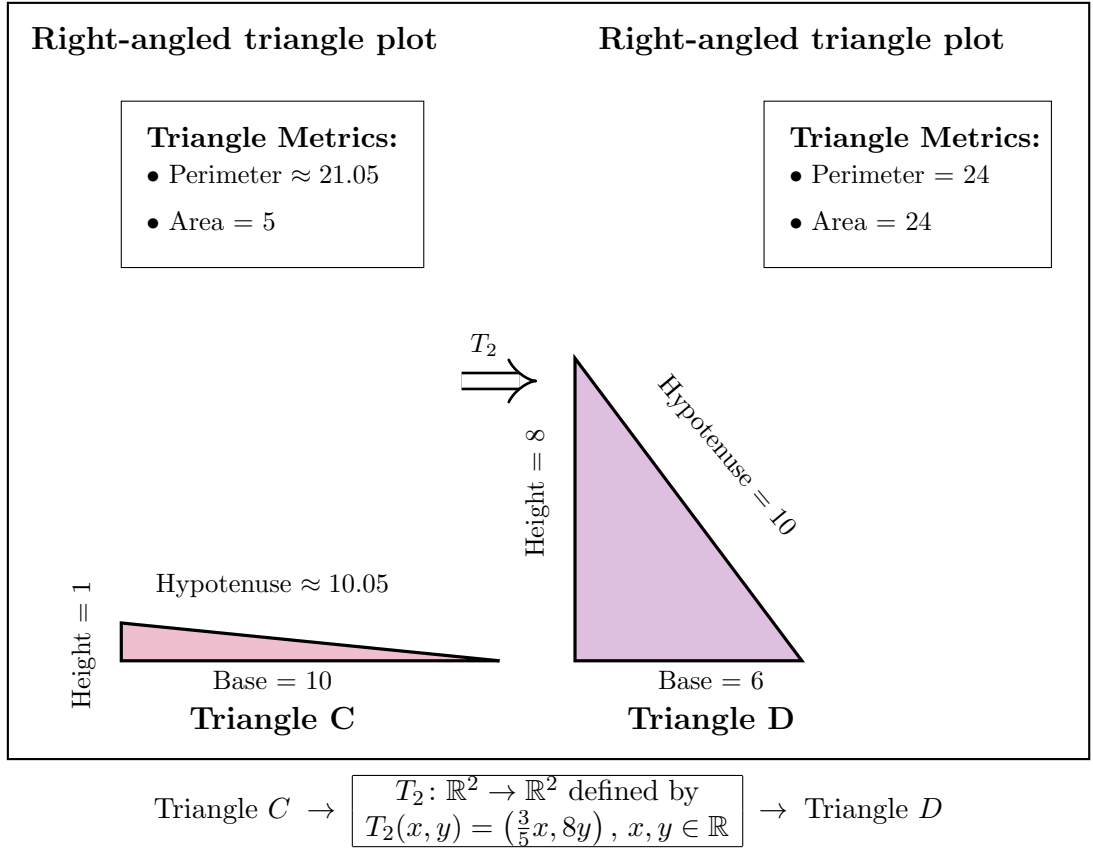



\section{Application to a nonhomogeneous linear parabolic partial differential equation}

This section deals with an application of our theorem to establish the existence and uniqueness of a solution to a nonhomogeneous linear parabolic partial differential equation satisfying a given initial condition. Consider the following problem:
\begin{align}\label{31}
\begin{cases}
u_t(x,t)=u_{xx}(x,t)-a u(x,t)+\mathcal{G}(x,t,u(x,t),u_x(x,t)), & \text{for } 
-\infty<x<\infty,  \\
&0<t\le T, a>0,  \\
u(x,0)=\varphi(x), & \text{if } -\infty<x<\infty.
\end{cases}
\end{align}

It is assumed that $\varphi$ is continuously differentiable such that $\varphi$ and $\varphi'$ are bounded and $\mathcal{G}$ is continuous. 

By a solution of the system (\ref{31}) we meant a function $u(x,t)$ defined on $\mathbb{R}\times I$, where $I=[0,T], T>0$, satisfying the following conditions:

(i) $u,u_t,u_x$ and $u_{xx}\in \mathcal{C}(\mathbb{R}\times I)$, where $\mathcal{C}(\mathbb{R}\times I)$ is the space of all continuous functions defined on $\mathbb{R}\times I$;

(ii) $u$ and $u_x$ are bounded in $\mathbb{R}\times I$;

(iii) $u_t(x,0)=u_{xx}(x,0)-a u(x,0)+\mathcal{G}(x,0,u(x,0),u_x(x,0))$ for $-\infty<x<\infty$;

(iv) $u(x,0)=\varphi(x)$ for $-\infty<x<\infty$.

Consider $$ \mathcal{B}=\{u(x,t):u,u_x\in \mathcal{C}(\mathbb{R}\times I)\text{ and }\|u\|^1<\infty\}, $$

where $$\|u\|^1=\sup_{(x,t)\in \mathbb{R}\times I}|u(x,t)|+\sup_{(x,t)\in \mathbb{R}\times I}|u_x(x,t)|. $$

Then $(\mathcal{B},\|\cdot\|^1)$ is a Banach space.

\begin{theorem}\label{31.5}
Now consider the problem $(\ref{31})$ with the following conditions:

$(a)$ For any $c>0$ with $|s|<c$ and $|p|<c$, the function $\mathcal{G}(x,t,s,p)$ is uniformly H\"older continuous in $x$ and $t$ for each compact subset of $\mathbb{R}\times I$;

$(b)$ There exists a constant
\begin{align}\label{32}
0<c_\mathcal{G}\le \frac{a}{\sqrt{a}+1}r, \text{ with }0<r<\frac{1}{\sqrt{2}} \text{ for all }(s_1,p_1),(s_2,p_2)\in \mathbb{R}^2 
\end{align} 
such that 
\begin{align}\label{33}
|\mathcal{G}(x,t,s_1,p_1)-\mathcal{G}(x,t,s_2,p_2)|\le c_\mathcal{G}[|s_1-s_2|+|p_1-p_2|].
\end{align}

$(c)$ $\mathcal{G}$ is bounded whenever $s$ and $p$ are bounded.\\
Then, the existence and uniqueness of the solution of the system $(\ref{31})$ is necessarily affirmative.
\end{theorem}

\begin{proof}
It is essential to note that the problem (\ref{31}) is equivalent to the integral equation 
\begin{align}\label{34}
u(x,t)&=\int_{-\infty}^\infty e^{-at} K(x-\xi,t)\varphi(\xi)d\xi+\nonumber\\
&\int_0^t \int_{-\infty}^\infty e^{-a(t-\tau)} K(x-\xi,t-\tau)\mathcal{G}(\xi,\tau,u(\xi,\tau),u_x(\xi,\tau))d\xi d\tau,
\end{align}
for all $-\infty<x<\infty$ and $0<t\le T$, where $K(x,t)=\frac{1}{\sqrt{4\pi t}}e^{-\frac{x^2}{4t}}$ for all $x\in \mathbb{R}$ and $t>0$.\\
For $-\infty<x<\infty$ and $0<t\le T$, define 
\begin{align*}
(Tu)(x,t)&=\int_{-\infty}^\infty e^{-at} K(x-\xi,t)\varphi(\xi)d\xi+\\
&\int_0^t \int_{-\infty}^\infty e^{-a(t-\tau)} K(x-\xi,t-\tau)\mathcal{G}(\xi,\tau,u(\xi,\tau),u_x(\xi,\tau))d\xi d\tau.
\end{align*}
Then 
\begin{align*}
(Tu)_x(x,t)&=\int_{-\infty}^\infty e^{-at} K_x(x-\xi,t)\varphi(\xi)d\xi+\\
&\int_0^t \int_{-\infty}^\infty e^{-a(t-\tau)} K_x(x-\xi,t-\tau)\mathcal{G}(\xi,\tau,u(\xi,\tau),u_x(\xi,\tau))d\xi d\tau.
\end{align*}
We see that
\begin{align}\label{35}
&|(Tu)(x,t)-(Tv)(x,t)|\nonumber\\
=&\left\lvert\int_0^t \int_{-\infty}^\infty e^{-a(t-\tau)} K(x-\xi,t-\tau)[\mathcal{G}(\xi,\tau,u(\xi,\tau),u_x(\xi,\tau))-\mathcal{G}(\xi,\tau,v(\xi,\tau),v_x(\xi,\tau))]d\xi d\tau\right\rvert\nonumber\\
\le & \int_0^t \int_{-\infty}^\infty e^{-a(t-\tau)} K(x-\xi,t-\tau)|\mathcal{G}(\xi,\tau,u(\xi,\tau),u_x(\xi,\tau))-\mathcal{G}(\xi,\tau,v(\xi,\tau),v_x(\xi,\tau))|d\xi d\tau\nonumber\\
\le & \int_0^t e^{-a(t-\tau)} \int_{-\infty}^\infty  K(x-\xi,t-\tau)[|u(\xi,\tau)-v(\xi,\tau)|+|u_x(\xi,\tau)-v_x(\xi,\tau)|]d\xi d\tau\nonumber\\
\le & c_\mathcal{G} \|u-v\|^1\int_0^t e^{-a(t-\tau)} \int_{-\infty}^\infty  K(x-\xi,t-\tau)d\xi d\tau\nonumber\\
\le & \frac{c_\mathcal{G}}{a} \|u-v\|^1
\end{align}
and
\begin{align}\label{36}
&|(Tu)_x(x,t)-(Tv)_x(x,t)|\nonumber\\
=&\left\lvert\int_0^t \int_{-\infty}^\infty e^{-a(t-\tau)} K_x(x-\xi,t-\tau)[\mathcal{G}(\xi,\tau,u(\xi,\tau),u_x(\xi,\tau))-\mathcal{G}(\xi,\tau,v(\xi,\tau),v_x(\xi,\tau))]d\xi d\tau\right\rvert\nonumber\\
\le & \int_0^t \int_{-\infty}^\infty e^{-a(t-\tau)} |K_x(x-\xi,t-\tau)||\mathcal{G}(\xi,\tau,u(\xi,\tau),u_x(\xi,\tau))-\mathcal{G}(\xi,\tau,v(\xi,\tau),v_x(\xi,\tau))|d\xi d\tau\nonumber\\
\le & \int_0^t e^{-a(t-\tau)} \int_{-\infty}^\infty  |K_x(x-\xi,t-\tau)|[|u(\xi,\tau)-v(\xi,\tau)|+|u_x(\xi,\tau)-v_x(\xi,\tau)|]d\xi d\tau\nonumber\\
\le & c_\mathcal{G} \|u-v\|^1\int_0^t e^{-a(t-\tau)} \int_{-\infty}^\infty  |K_x(x-\xi,t-\tau)|d\xi d\tau\nonumber\\
\le & \frac{c_\mathcal{G}}{\sqrt{a}} \|u-v\|^1.
\end{align}
Therefore,
\begin{align}\label{37}
\|Tu-Tv\|^1=&\sup_{(x,t)\in \mathbb{R}\times I}|(Tu)(x,t)-(Tv)(x,t)|+\sup_{(x,t)\in \mathbb{R}\times I}|(Tu)_x(x,t)-(Tv)_x(x,t)|\nonumber\\
&\le \frac{c_\mathcal{G}}{a} \|u-v\|^1+\frac{c_\mathcal{G}}{\sqrt{a}} \|u-v\|^1\nonumber\\
&=c_\mathcal{G}\left(\frac{1}{a}+\frac{1}{\sqrt{a}}\right)\|u-v\|^1\nonumber\\
&=r \|u-v\|^1.
\end{align}
This shows that $T$ is a Banach contraction with contraction constant $r<\frac{1}{\sqrt{2}}$, that is, $T$ is a mapping contracting hypotenuse of a right-angled triangle due to Proposition \ref{22}. Also $T$ satisfies all the conditions of Theorem \ref{25} and Theorem \ref{26} and therefore $T$ has a unique fixed point. Hence (\ref{34}) has a unique solution, equivalent to the unique solution of the given partial differential equation (\ref{31}).
\end{proof}

\begin{example}\label{38}
Consider the following nonhomogeneous linear parabolic partial differential equation:
\begin{align}\label{e31}
\begin{cases}
u_t(x,t)=u_{xx}(x,t)-8 u(x,t)+\sqrt{|u(x,t)|^2+|u_x(x,t)|^2}, & \text{for } \text{$-\infty<x<\infty$, $0<t\le \frac{5\pi}{2}$},  \\
u(x,0)=x^2+\sin(\pi x)+\cos\left(\frac{\pi}{2}x\right), & \text{if } -\infty<x<\infty.
\end{cases}
\end{align}
Then $\mathcal{G}(x,t,u(x,t),u_x(x,t)=\sqrt{|u(x,t)|^2+|u_x(x,t)|^2}$ and
\begin{align*}
&|\mathcal{G}(x,t,u(x,t),u_x(x,t))-\mathcal{G}(x,t,v(x,t),v_x(x,t))|\\
=&|\sqrt{|u(x,t)|^2+|u_x(x,t)|^2}-\sqrt{|v(x,t)|^2+|v_x(x,t)|^2}|\\
\le & \sqrt{|u(x,t)-v(x,t)|^2+|u_x(x,t)-v_x(x,t)|^2}\\
\le & |u(x,t)-v(x,t)|+|u_x(x,t)-v_x(x,t)|.
\end{align*}
Showing that $c_\mathcal{G}=1$ and we see that for $r=\frac{1}{2}$, $c_\mathcal{G}\le \frac{8}{2\sqrt{2}+1}\times\frac{1}{2}$. Therefore, all the conditions of Theorem \ref{31.5} are satisfied and (\ref{e31}) possesses a unique solution.
\end{example}

\section{Conclusion}
In this article, mappings contracting hypotenuse of a right-angled triangle are explored in the context of metric spaces. A sufficient condition is established for the existence of fixed points of these mappings. It is observed that the absence of periodic points of prime period 2 is necessary to ensure the existence of fixed points. These mappings can attain at most two fixed points. Thus, a necessary condition is derived for the uniqueness of the fixed point. Subsequently, a geometric interpretation of the proposed mapping is provided to establish its distinction from both perimeter contracting and area contracting mappings through pictorial representation. In addition, illustrative examples are provided to support the theoretical findings. Next, we investigate the Ulam-Hyers stability of the fixed point equation related to newly introduced mappings. Finally, we apply the theoretical findings to establish the existence of a solution for a nonhomogeneous linear parabolic partial differential equation.


\subsection*{Availability of data and materials}
Not applicable.

\subsection*{Competing interests}
The authors declare that they have no competing interests.

\subsection*{Funding}
Not applicable.

\subsection*{Author's contributions}
All authors contributed equally and significantly in writing this paper.
All authors read and approved the final manuscript.

\bibliographystyle{plain}

\begin{thebibliography}{00}

\bibitem{B22}
S.~Banach.
\newblock Sur les op\'erations dans les ensembles abstraits et leur application aux \'equations int\'egrales.
\newblock {\em Fund. Math.}, 3:133--181, 1922.

\bibitem{BMD2}
A.~Banerjee, P.~Mondal, and L.~K. Dey.
\newblock Perimetric contractions of Kannan and Chatterjea type in quadrilaterals.
\newblock {\em To appear in Mat. Vesn.}, 2025.
	
\bibitem{BMD1}
A.~Banerjee, P.~Mondal, and L.K. Dey.
\newblock Perimetric contraction on quadrilaterals.
\newblock {\em Science {\&} Technology Asia}, 30(4), 55--67, 2025.

\bibitem{BMD3}
A.~Banerjee, P.~Mondal, and L.K. Dey.
\newblock On generalized Chatterjea type mappings in interpolative metric spaces.
\newblock {\em communicated}, 2025.

\bibitem{BMDS26}
A.~Banerjee, P.~Mondal, L.K. Dey, and W. Sintunavarat.
\newblock A perspective on Banach and Kannan mappings contracting the perimeter of triangles in interpolative metric spaces.
\newblock {\em Nonlinear Anal. Model. Control, 31, pp. 1–21. doi:10.15388/namc.2026.31.47663.}, 2026.

\bibitem{BPS25}
C.~Bey, E.~Petrov, and R.~Salimov.
\newblock On three-point generalizations of {B}anach and {E}delstein fixed point theorems.
\newblock {\em Filomat}, 39(1):185--195, 2025.

\bibitem{BP24}
R.K.~Bisht and E.~Petrov.
\newblock Three point analogue of {Ć}irić-{R}eich-{R}us type mappings with non-unique fixed points.
\newblock {\em J. Anal.}, 32(5):2609-2627, 2024.

\bibitem{H41}
D.H.~Hyers.
\newblock On the stability of the linear functional equation.
\newblock {\em Proc. Natl. Acad. Sci. USA}, 27(4):222-224, 1941.

\bibitem{JPS25}
M.~Jleli, C.M.~P{\u{a}}curar and B. Samet.
\newblock Fixed point results for contractions of polynomial type.
\newblock {\em Demonstr. Math.}, 58(1):20250098, 2025.

\bibitem{PP24}
C.M. P{\u{a}}curar and O.~Popescu.
\newblock Fixed point theorem for generalized {C}hatterjea type mappings.
\newblock {\em Acta. Math. Hungar.}, 1-10, 2024.

\bibitem{P23}
E.~Petrov.
\newblock Fixed point theorem for mappings contracting perimeters of triangles.
\newblock {\em J. Fixed Point Theory Appl.}, 25:74, 2023.

\bibitem{PB24}
E.~Petrov and R.K. Bisht.
\newblock Fixed point theorem for generalized {K}annan type mappings.
\newblock {\em Rend. Circ. Mat. Palermo, II. Ser}, 73(2024), 2895--2912, 2024.

\bibitem{R26a}
K.~Roy.
\newblock Applications of a mapping contracting area of geometrical figures.
\newblock {\em The Journal of Analysis}, 34(2):1197 - 1214, 2026.

\bibitem{R24}
K.~Roy.
\newblock Mappings contracting axes of ellipse.
\newblock {\em The Journal of Analysis}, 32(6):3557--3563, 2024.

\bibitem{R26b}
K.~Roy.
\newblock Mappings contracting transverse axis of hyperbola.
\newblock {\em Appl. Gen. Topol.}, 27(1):23621, 2026.
\end{thebibliography}

\end {document}